\documentclass[12pt]{amsart}
\usepackage{verbatim}
\usepackage{textcomp}
\usepackage{graphicx,color,bbm, bbold}
\usepackage{amsmath,amsfonts,amssymb,mathrsfs,stackrel}

\definecolor{Red}{cmyk}{0,1,1,0}

\definecolor{verde}{cmyk}{1,0,1,0}

\definecolor{azul}{cmyk}{1,1,0,0}

\numberwithin{equation}{section}

\newcommand{\eqd}{\stackrel{\tiny d}{=}}
\newcommand{\Exp}{\mathop{\mathrm{Exp}}}

\def\Ed{{\mathbb{E}}}

\newcommand{\N}{\mathbb{N}}

\newcommand{\R}{\mathbb{R}}

\renewcommand{\P}{\mathbb{P}}

\newcommand{\dtj}{d_t(j)}

\newcommand{\dttj}{d_{t+1}(j)}

\newcommand{\dmaxt}{d_{max}(G_t)}

\newcommand{\Itj}{\mathbb{1}_{\{T_{j,1}=t_0\}}}

\renewcommand{\a}{\alpha}
\renewcommand{\b}{\beta}
\newcommand{\g}{\gamma}
\renewcommand{\d}{\delta}
\newcommand{\e}{\varepsilon}

\renewcommand{\l}{\lambda}

\newcommand{\D}{\Delta}

\newcommand{\GG}{\mathcal{F}}
\newcommand{\G}{\Gamma}

\newcommand{\be}{\begin{equation}}
\newcommand{\ee}{\end{equation}}

\newtheorem{teorema}{Theorem}
\newtheorem{proposicao}[equation]{Proposition}
\newtheorem{definition}{Definition}
\newtheorem{lema}{Lemma}
\newtheorem{corolario}[equation]{Corollary}

\usepackage{lipsum}

\begin{document}
\title{Large Communities in a scale-free network}
\author{Caio Alves$^1$}
\address{$^1$ Department of Statistics, Institute of Mathematics,
 Statistics and Scientific Computation, University of Campinas --
UNICAMP, rua S\'ergio Buarque de Holanda 651,
13083--859, Campinas SP, Brazil\newline
e-mail: {\itshape \texttt{narrowstreets@gmail.com}}}

\author{R{\'e}my Sanchis$^2$}
\address{$^2$Departamento de Matem{\'a}tica, Universidade Federal de Minas Gerais, Av. Ant\^onio
Carlos 6627 C.P. 702 CEP 30123-970 Belo Horizonte-MG, Brazil
\newline
e-mail: {\itshape \texttt{rsanchis@mat.ufmg.br}}}

\author{Rodrigo Ribeiro$^3$}
\address{$^3$Departamento de Matem{\'a}tica, Universidade Federal de Minas Gerais, Av. Ant\^onio
Carlos 6627 C.P. 702 CEP 30123-970 Belo Horizonte-MG, Brazil
\newline
e-mail: {\itshape \texttt{rodrigo-matematica@ufmg.br}}}

\begin{abstract}
We prove the existence of a large complete subgraph w.h.p.\ in a preferential attachment random graph process with an \emph{edge-step}. That is, we consider a dynamic stochastic process for constructing a graph in which at each step we independently decide, with probability $p\in(0,1)$, whether the graph receives a new vertex or a new edge between existing vertices. The connections are then made according to a preferential attachment rule.  We prove that the random graph $G_{t}$ produced by this so-called GLP \textit{(Generalized linear preferential)} model at time~$t$ contains a complete subgraph whose vertex set cardinality is given by~$t^\a$, where $\a = (1-\e)\frac{1-p}{2-p}$, for any small~$\e>0$ asymptotically almost surely. 
\end{abstract}

\maketitle

\section{Introduction}
In recent years, the popularization of computers and of the Internet made possible the analysis of large ammounts of data. The empiric investigation of real-world complex networks, such as the WWW, the network of collaboration and neural networks \cite{bamodel,Dorogovtsev,aNewman2003} has shown that these networks exhibit a distinct behaviour from the classical Erd\"{o}s-R\'{e}nyi model of random graphs (see e.g. \cite{Alon00theprobabilistic} for the definition of the model). In particular, the distribution of the degrees of the vertices from these real-life networks  obeys a power-law, the so-called \textit{scale-free} phenomenon.

After these findings, the scientific community made significant efforts to explain the origin of the power-law phenomenon and to construct models capable of capturing the properties presented by the empirical evidences. One possible explanation is the phenomenon of \textit{Preferential Attachment} (PA), suggested in \cite{bamodel}, which states that new individuals in the networks preferentially connect themselves to the more popular ones. Nowadays there exists a broad literature about PA models; see \cite{bollrior,chungbook,cooper,durretbook,remco} and  references therein.

Beyond the \textit{power-law} phenomenon, other important questions arose, some of them about the spread of diseases in scale-free graphs and the vulnerability of such graphs to deliberate attack; see~\cite{bollobasRobus} for an example. In the context of vulnerability, cliques - i.e., complete subgraphs - play a significant role. When the attack is completely random, large cliques have high probability of remaining connected. On the other hand, deliberate attacks directed towards them represent a threat to the network's connectedness. Still in the practical context, the presence of certain subgraphs in biological networks, called \textit{motifs}, is related to functional properties selected by evolution \cite{motifs}. 


Furthermore, the order of the largest complete subgraph  provides a lower bound to the number of triangles in~$G$, a fundamental  quantity to study the so-called \textit{global clustering coefficient} of $G$ (see e.g.  \cite{bollrior,eggemann,russos,russos_gc}). For more works related to cliques in scale-free random graphs, see also \cite{cliquesbianconi,communitystructure,JLN}.

In this paper, we investigate a random graph model in a class known as GLP \textit{(Generalized linear preference)} in the literature of Computer Science and Physics. Many results are known about a variety of random graph models belonging to the GLP class. In \cite{mori}, the author proves convergence results for the maximum degree of a random graph model in the GLP class. In \cite{eggemann} the authors stablish the decay speed of the expected value of the global clustering coefficient. In \cite{remcodiam}, the authors find lower and upper bounds for the diameter of a subclass of GLP, demonstrating that they also capture the \textit{small-world} phenomenon. 
In \cite{degcorr} the authors analyze degree correlation in many models, including the GLP class and compare the theoretical results with networks obtained from empirical data.

The model here investigated is a modification of the PA model in which links between existing vertices are allowed. The effect of this alteration has positive consequences. In \cite{chungbook} the authors prove that this model obeys a power law with an tunable exponent. Empirically, the model also has shown some advantages over other models. In \cite{evnmodels} a statistical analysis is made comparing real world prediction capabilities between this GLP model and other influential network models, such as the Erd\"{o}s-R\'{e}nyi, Albert-Barabasi and Tel Aviv Network Generator. The results suggest that the GLP process we study in this paper outperforms these popular models when the task is predicting or mimicking real-world complex networks.

Let us briefly describe the process. This model has two parameters: a real number $p\in[0,1]$ and an initial graph $G_0$. 
For the sake of simplicity we will consider $G_0$ to be the graph with one vertex and one loop. We consider the following two stochastic operations that can be performed on the graph $G$:
\begin{itemize}
  \item \textit{Vertex-step} - Add a new vertex $v$, and add an edge $\{u,v\}$ by choosing $u\in G$ with probability proportional to its degree. \\
  
  \item \textit{Edge-step} - Add a new edge $\{u_1,u_2\}$ by independently choosing vertices $u_1,u_2\in G$ with probability proportional to their degrees. We note that we allow loops to be added, and we also allow a new connection to be added between vertices that already shared an edge.

\end{itemize}
We consider a sequence $(Z_t)_{t\ge 1}$ of i.i.d random variables such that $Z_t\eqd \mathrm{Ber}(p)$. We define inductively a random graph process $(G_t)_{t \ge 0}$  as follows: start with~$G_0$, the graph with one vertex and one loop. Given $G_{t}$, form $G_{t+1}$ by performing a \textit{vertex-step} on $G_t$ when $Z_t=1$, and performing an \textit{edge-step} on $G_t$ when $Z_t=0$. The resulting process is the object of study of this paper.

The goal of this paper is to investigate the existence and size of large complete subgraphs in $G_t$, which we  refer to as \textit{communities}. We are interested in communities whose vertex set cardinality, which we also call the community's \emph{order}, goes to infinity as the process evolves. With this in mind, we prove the following result:

\begin{teorema}[The existence of a large community]\label{teo:community} For any $\e>0$, the graph $G_{2t}$ has a complete subgraph of order~$t^{(1-\e)\frac{(1-p)}{2-p}}$, asymptotically almost surely.
\end{teorema}

\textbf{Remark. } 
Making use of an ansatz regarding decorrelation inequalities between the random variables that count the number of edges between predefined pairs of vertices, one can  show that the expectation of the number of triangles of~$G_t$ has order $t^{3\frac{(1-p)}{2-p}}\log^2t$. This indicates that the largest clique should be smaller than $t^{(1+\e)\frac{(1-p)}{2-p}}$ for any $\e>0$. We also remark that, putting $p=1$, the expected order of triangles is $\log^2t$, a result that is consistent with \cite{bollrior}.

\textbf{Main ideas and organization.} Our analysis requires, as a first step, upper bounds on the vertices' degree. The proofs we give in Section \ref{sec:upper} for these bounds follow the standard arguments involving martingales and Azuma's inequality to guarantee measure concentration. We also need lower bounds for the vertices' degrees, more specifically, we need the presence of a large number of vertices having very high degree, which is proven in Section \ref{sec:lower}. Unfortunately, in this direction Azuma's inequality leads only to trivial lower bounds. This is due to the fact that a single vertex may not increase its degree for a long a time with a non negligible probability, which obstructs a concentration result as strong as the one given by the upper bound. 

To overcome the above issue, we keep track of the random time in which a vertex achieves a specific degree~$k$. The main idea here is to identify sets of $m$ consecutive vertices. This proceedure justifies the intuitive feeling that, unlike a single vertex, blocks of vertices are more stable in the sense that they do not take very long to achieve a desirable high degree. We formalize this intuition in Lemma \ref{lema:major} which gives an upper bound for the tail of these random times. The proof follows the idea of Lemma $3.1$ of \cite{mori}, which consist in proving a stochastic domination of these random times by a function of independent exponential distributed random variables. Unfortunately, in our case loops may occur due to the edge-step and this fact not only prevents a straightforward application of this lemma but also makes  its generalization harder. 

With the aid of Lemma \ref{lema:major} we stablish a lower bound for the degree of blocks in Theorem \ref{teo:cotainf}. A direct corollary of this theorem is the presence at time $t$ of many vertices having degree greater than~$\sqrt{t}$. This is a fundamental step in order to prove Theorem \ref{teo:community}, which is done in Section \ref{sec:comu}. The idea is that pairs of  vertices with very high degree in $G_t$ cannot remain disconnected until time $2t$ with non negligible probability. Another consequence of Theorem \ref{teo:cotainf}, which we do not make use of here but we must point out, is the existence, w.h.p.,  of a vertex having degree of order close to the expected maximum degree $t^{1-p/2}$. Vertices of high degree play an important roles in the analysis of complex networks. They offer a lower bound for the number of paths of size two, which is important for the calculation of the \textit{global clustering}, and they prove themselves useful for upper bounds for the diameter since they tend to attract more connections; see Section 3 of \cite{remcodiam} for an example.


\section{Upper bound for the degree} \label{sec:upper}

In this section we will establish a simple bound on the probability that a given vertex has a large degree. The result will follow from a application of Azuma's inequality (see Theorem~$2.19$ from~\cite{chungbook}).

Given two positive integers, $i$ and $j$, we will make a slight abuse of notation and let these numbers also  denote respectively the $i$-th  and $j$-th vertices to be added to the process $(G_t)_{t\geq 0}$. The random time in which the $j$-th vertex is added to the graph will be denoted by $T_{j,1}$. Given a vertex $v$, we will let $d_t(v)$ denote the degree of $v$ in $G_t$. 
\begin{definition}
Since the constant $1-p/2$ is going to appear many times throughout this paper, it deserves a special notation. We write 
\begin{equation*}
c_p := 1-p/2.
\end{equation*}
\end{definition}
\begin{proposicao}\label{prop:graumartingal} For each vertex $j$ and each $t_0 \ge j$ the sequence of random variables~$\left(Z_t \right)_{t \ge t_0}$ defined as
\begin{equation}
\label{e:Xtdef}
 Z_t : =\frac{\dtj \Itj}{\prod_{s=1}^{t-1}\left( 1+ \frac{c_p}{s} \right)}
\end{equation}
is a martingale starting from $t_0$.
\end{proposicao}

\begin{proof} We consider the process $(G_t)_{t\geq 0}$ to be adapted to a filtration $(\GG_t)_{t\geq 1}$. Define $\D \dtj := \dttj - \dtj$. It is clear that $\D \dtj \in \{0,1,2\}$. Furthermore, conditioned on $\GG_t$, we know the probability that  $\D \dtj $ takes each of these values. So that, assuming that $j$ already exists at time $t$, we have
\begin{equation*}
\begin{split}
\Ed \left[\D \dtj \middle| \GG_t \right] & = 1\cdot p \frac{\dtj}{2t} + 1\cdot(1-p)2\frac{\dtj}{2t}\left(1 - \frac{\dtj}{2t} \right) + 2\cdot(1-p)\frac{(\dtj)^2}{4t^2}
 \\
& = c_p\frac{\dtj}{t}.
\end{split}
\end{equation*}
The information ``vertex $j$ exists at time $t$" can be introduced in the equation using the random variable $\Itj$. Using the fact that~$\Itj$ is $\GG_{t_0}$ measurable, we gain
\begin{equation}\label{eq:expdeg}
\begin{split}
\Ed \left[\dttj \Itj \middle| \GG_t \right] = \left(1+\frac{c_p}{t}\right)\dtj \Itj.
\end{split}
\end{equation}
Dividing the above equation by $\prod_{s=1}^{t}\left( 1+ \frac{c_p}{s} \right)$ we obtain the desired result.
\end{proof}

Now we prove the main result of the section.
\begin{proposicao}[Upper bound for the degree] \label{prop:upperboundeg}There exists a universal positive constant~$C_1$, such that for every vertex $j$ we have that
\[
\P \left( \dtj \ge C_1t^{c_p}\sqrt{\frac{\log(t)}{j^{1-p}}} \right) \le \frac{1}{t^{100}}.
\]
\end{proposicao}

\begin{proof} We define $\phi(t) := \prod_{s=1}^{t-1}\left( 1+ \frac{c_p}{s} \right)$. Note that we can write 
\begin{equation*}
\frac{\dtj}{\phi(t)}= \sum_{t_0=j}^t \frac{\dtj}{\phi(t)} \mathbb{1}_{\{ T_{j,1} = t_0\}},
\end{equation*}
and by Proposition~\ref{prop:graumartingal} each term in the sum is a martingale. We want to apply Azuma's inequality for each summand, but first we need some bounds on~$\phi$. We get the asymptotic behaviour of $\phi$ by noting that we can express its rule by a ratio of Gamma functions, as follows by using the Gamma function's duplication property:
\[
\phi(t) = \prod_{s=1}^{t-1}\left( 1+ \frac{c_p}{s} \right) = \frac{\G(t+c_p)}{\G(1+c_p)\G(t)}.
\]
And, by property $6.1.46$ of~\cite{handbook}, $\phi(t) \sim t^{c_p}$. This means that, for some constant $c_1>0$, $\phi(t) > c_1t^{c_p}$.

In order to apply Azuma's inequality, we must bound the variation of the random variable~$Z_t$ defined in~\eqref{e:Xtdef}, which satisfies the following upper bound
\begin{equation}
a_t:=\left| \frac{d_{t+1}(j) \mathbb{1}_{\{T_{j,1} = t_0 \}}}{\phi(t+1)} - \frac{\dtj \mathbb{1}_{\{T_{j,1} = t_0 \}}}{\phi(t)} \right|  \le \left|\frac{ \dttj - \left( 1 + \frac{c_p}{t} \right)\dtj }{\phi(t+1)}\right|
\le \frac{2+c_p}{\phi(t+1)},
\end{equation}
since $\D \dtj \le 2$ and $\dtj \le 2t$. By the above discussion about $\phi$ we have that $a_t^2 < \frac{(2+2c_p)^2}{c_1^2 t^{2-p}}$, which implies $\sum_{s=t_0}^t a_s^2 < c' t_0^{-(1-p)}$ for some constant $c'>0$. Then, applying Azuma's inequality, we obtain
\begin{equation}
\P\left(\left| \frac{\dtj \mathbb{1}_{\{T_{j,1} = t_0 \}}}{\phi(t)} - \Ed\left[\frac{d_{t_0}(j) \mathbb{1}_{\{T_{j,1} = t_0 \}}}{\phi(t_0)}\right] \right| > \l \right) \le 2\exp\left(-\l^2t_0^{1-p}/2c'\right).
\end{equation}
Note that
\[
\Ed\left[\frac{d_{t_0}(j) \mathbb{1}_{\{T_{j,1} = t_0 \}}}{\phi(t_0)}\right] =\phi(t_0)^{-1} \P ( T_{j,1} = t_0 )\leq \phi(t_0)^{-1}.
\]
By choosing $\l = c_2\sqrt{t_0^{p-1}\log(t)}$, where~$c_2$ is a sufficiently large positive constant depending only on $p$, we gain
\[
\P\left( \dtj\mathbb{1}_{\{T_{j,1} = t_0 \}} \ge c_2 t^{c_p}\sqrt{\frac{\log(t)}{t_0^{1-p}}}+\frac{\phi(t)}{\phi(t_0)}\right) \le \frac{1}{t^{101}}.
\] 
Using the union bound and the asymptotic behaviour of~$\phi$, we obtain the result for all possibles times $t_0\ge j$: there exists a positive constant~$C_1$ depending only on~$p$ such that
\[
\P\left(\bigcup_{t_0 = j}^t \left\lbrace\dtj\mathbb{1}_{\{T_{j,1} = t_0 \}} \ge C_1t^{c_p}\sqrt{\frac{\log(t)}{j^{1-p}}}\right\rbrace \right) \le \frac{1}{t^{100}}.
\] 
This finishes the proof.
\end{proof}
Using the union bound again, we obtain the following result:
\begin{corolario}[Upper bound for the maximum degree]\label{cor:maxdeg} Let $\dmaxt$ denote the maximum degree among all the vertices of $G_t$. Then, there exists a universal positive constant $C_2$ such that
\[
\P\left(\dmaxt \ge C_2t^{c_p}\sqrt{\log(t)}\right) \le \frac{1}{t^{99}}.
\]
\end{corolario}
\section{Lower bounds for the degree}\label{sec:lower}
%
%
This section is devoted to proving the results needed to state a useful lower bound on the degree of the vertices that entered the random graph early in the history of the process. We prove two lemmas that let us control the tail of the random times  $T_{j,k}^{(m)}$, defined below,  before proving the main result, Theorem~\ref{teo:cotainf}. First we need a new notation:
\begin{definition} Fix a vertex $j$  and two integers $m,k \ge 1$. We define the random time
\[
T^{(m)}_{j,k} :=  \inf_{t\ge 1}\left\lbrace \sum_{i=(j-1)m+1}^{jm}d_t (i) = k \right\rbrace.
\]
\end{definition}
We also write $T_{j,k}:=T^{(1)}_{j,k}$. In other words, $T_{j,k}$ is the first time that the $j$-th vertex has degree at least $k$. $T^{(m)}_{j,k}$ can then be explained in the following way: assume that we identify all the vertices $1$ through $m$, then identify all the vertices $m+1$ through $2m$ and so on. We let $d_{t,m}(j)$ denote the degree of the $j$-th block of $m$ vertices. Then $T^{(m)}_{j,k}$ is the first time that~$d_{t,m}(j)$ is larger than, or equals to, $k$. 
\begin{lema}\label{lema:major} Let $0<\gamma<(c_p^{-1}-1)$. For large enough $m \in \N$ and $k \ge m$ and $j \ge(m^{2/(1-p)}+1)$, we can construct a sequence~$\eta_m, .. , \eta_{k}$ of independent random variables, with 
\begin{equation}
\label{eq:defeta}
\eta_i \eqd \Exp{\Big(c_p\Big(1-\frac{1-p}{2(2-p)i^{\gamma}}\Big)i \Big)},\text{ for }i=m,\dots,k,
\end{equation}
such that the whole sequence is independent of $T^{(m)}_{j,m},..,T^{(m)}_{j,k+1}$, and such that
\[
\P \left( T^{(m)}_{j,k+1} > t \right) \le \P \left( T^{(m)}_{j,m}\exp \left( \sum_{i = m}^k \eta_i \right) > t \right) + \frac{m}{[(j-1)m]^{99}}
\]
\end{lema}
\begin{proof} We follow the idea of the proof of Lemma $3.1$ in \cite{mori}. But in our context the existence of the \textit{edge-step} demands more attention and prevents a straightforward application of this lemma.

We begin by constructing the $k+1-m$ independent random variables $\eta_m, .. , \eta_{k}$ with distribution given by \eqref{eq:defeta}, the whole sequence being independent of the random times $T^{(m)}_{j,m},..,T^{(m)}_{j,k+1}$. Observe that
\begin{equation}
\label{eq:tjn}
\begin{split}
\P \left( T_{j,k+1}^{(m)} > t \right) & = \sum_{s=k}^{\infty} \P \left( T_{j,k+1}^{(m)} > t \middle| T_{j,k}^{(m)} = s\right)\P \left( T_{j,k}^{(m)} = s \right) \\ 
& = \sum_{s=k}^{k^{1+\gamma}} \P \left( T_{j,k+1}^{(m)} > t \middle| T_{j,k}^{(m)} = s\right)\P \left( T_{j,k}^{(m)} = s \right) \\
&\quad + \sum_{s=k^{1+\gamma}}^{\infty} \P \left( T_{j,k+1}^{(m)} > t \middle| T_{j,k}^{(m)} = s\right)\P \left( T_{j,k}^{(m)} = s \right) \\ 
& \le \P \left( T_{j,k}^{(m)} \le k^{1+\gamma} \right) + \sum_{s=k^{1+\gamma}}^{\infty} \P \left( T_{j,k+1}^{(m)} > t \middle| T_{j,k}^{(m)} = s\right)\P \left( T_{j,k}^{(m)} = s \right).
\end{split}
\end{equation}
We obtain an upper bound for the term $\P \left( T_{j,k+1}^{(m)} > t \middle| T_{j,k}^{(m)} = s\right)$ in the following way: once the vertex (block) $j$ reaches degree $k$ at time $s$, we must avoid choosing $j$ at all the subsequent steps until time $t$. We note that there exists the possibility that $T_{j,k+1}^{(m)}=T_{j,k}^{(m)}$, in the case that we add a loop to $j$ at time $T_{j,k}^{(m)}$, but in this case $\P \left( T_{j,k+1}^{(m)} > t \middle| T_{j,k}^{(m)} = s\right)$ is equal to $0$, and our calculations remain the same. Noting that at each step~$r+1$ we choose the vertex~$j$ with probability 
\[
c_p\frac{d_r(j)}{r} - \frac{(1-p)d_r^2(j)}{4r^2},
\]
and recalling that $c_p:=1-p/2$, we obtain, for $s\geq k^{1+\gamma}$,
 \begin{equation} \label{eq:prod}
 	\begin{split}
    \P \left( T^{(m)}_{j, k + 1} > t \middle| T^{(m)}_{j, k} = s \right) & \le \prod_{r = s}^{t - 1} \left( 1 -
    \frac{c_p k}{r} + \frac{(1-p)k^2}{4r^2} \right) \\
    & = \prod_{r = s}^{t - 1} \left[ 1 -
    \frac{c_p k}{r}\left(1 - \frac{(1-p)k}{2(2-p)r} \right) \right] \\
    & \le \prod_{r = s}^{t - 1} \left[ 1 -
    \frac{c_p k}{r}\left(1 - \frac{(1-p)}{2(2-p)k^{\gamma}} \right) \right].
    \end{split}
  \end{equation}
We introduce the notation
\begin{equation}
\d_k:=\frac{(1-p)}{2(2-p)k^{\gamma}}.
\end{equation} 
Observe that 
\begin{equation*}
1- \frac{c_pk(1-\d_k)}{r} \le \exp{\left( \frac{1}{r}\right)}^{-c_pk(1-\d_k)}
\le \left( 1 + \frac{1}{r}\right)^{-c_pk(1-\d_k)}
 = \left( \frac{r}{r+1}\right)^{c_pk(1-\d_k)}.
\end{equation*}
Plugging the above inequality into (\ref{eq:prod}), noting that this results in a telescopic product, and recalling the definition of $\eta_i$, we get
\[  \P \left( T^{(m)}_{j, k + 1} > t \middle| T^{(m)}_{j, k} = s \right) \le \left( \frac{s}{t} \right)^{c_p(1-\d_k)k} = \P\left(T_{j,k}^{(m)} e^{\eta_k} > t \middle| T_{j,k}^{(m)} = s \right) .
\]
Combining the above inequality with (\ref{eq:tjn}), we obtain
\begin{equation} \label{eq:picachu}
\begin{split}
\P \left( T_{j,k+1}^{(m)} > t \right) & \le \P \left( T_{j,k}^{(m)} \le k^{1+\gamma} \right) + \sum_{s=k^{1+\gamma}}^{\infty} \P\left(T_{j,k}^{(m)} e^{\eta_k} > t \middle| T_{j,k}^{(m)} = s \right) \P \left( T_{j,k}^{(m)} = s \right) \\
& \le \P \left( T_{j,k}^{(m)} \le k^{1+\gamma} \right) + \sum_{s=k}^{\infty} \P\left(T_{j,k}^{(m)} e^{\eta_k} > t \middle| T_{j,k}^{(m)} = s \right) \P \left( T_{j,k}^{(m)} = s \right) \\
& = \P \left( T_{j,k}^{(m)} \le k^{1+\gamma} \right) + \P\left(T_{j,k}^{(m)} e^{\eta_k} > t \right).
\end{split}
\end{equation}
Write 
\[
err(k) := \P \left( T_{j,k}^{(m)} \le k^{1+\gamma} \right).
\]
By the above equation, we also have, recalling that $\eta_k$ is independent from both $T_{j,k-1}^{(m)}$ and~$\eta_{k-1}$,
\begin{equation}
\begin{split}
\P\left(T_{j,k}^{(m)} e^{\eta_k} > t \right) & = \int_{0}^{\infty} \P\left(T_{j,k}^{(m)} > \frac{t}{s} \right)\P\left(e^{\eta_k} = \mathrm{d} s \right) \\
& \le \int_{0}^{\infty} \left[\P\left(T_{j,k-1}^{(m)} e^{\eta_{k-1}} > \frac{t}{s} \right) + err\left(k-1\right)\right]\P\left(e^{\eta_k} = \mathrm{d}s \right) \\
& \le \P\left(T_{j,k-1}^{(m)} e^{(\eta_k + \eta_{k-1})} > t \right) + err(k-1),
\end{split}
\end{equation}
where $\P\left(e^{\eta_k} = \mathrm{d}s \right)$ denotes the measure in $\R$ induced by the random variable $e^{\eta_k}$. Proceeding in this way, we obtain
\[
\P \left( T_{j,k+1}^{(m)} > t \right) \le \P\left(T_{j,m}^{(m)}\exp\Bigg(\sum_{i=m}^{k}\eta_i\Bigg) > t \right) + \sum_{n=m}^k err(n).
\]
It remains to be shown that the sum of errors is sufficiently small. First we note that: 
\[
\left\lbrace T_{j,n}^{(m)} \le n^{1+\gamma} \right\rbrace = \left\lbrace \sum_{i=m(j-1)+1}^{mj} d_{n^{1+\gamma}}(i) \ge n \right\rbrace.
\]
But for small $n$ and large $j$ the above event is actually empty, since none of the $m$ vertices in the $j$-th block has enough time to be added by the process. In order to the above event to be non-empty, we need at least one of the random variables $d_{n^{1+\gamma}}(i)$, for $i \in \{ (j-1)m+1,\dots,jm\}$ to be possibly not identically null. For this, $n$ and~$j$ must satisfy the inequality below: 
\[
n^{1+\gamma} \ge (j-1)m+1 \iff n \ge \left[ (j-1)m+1\right]^{\frac{1}{1+\gamma}}.
\]
A straightforward application of Dirichlet's pigeon-hole principle shows that 
\[
\left\lbrace T_{j,n}^{(m)} \le n^{1+\gamma} \right\rbrace \subset \bigcup_{i=(j-1)m+1}^{jm} \left\lbrace d_{n^{1+\gamma}}(i)\ge \frac{n}{m}\right\rbrace.
\]
We note that, if
\begin{equation}
\label{e:ifjnm}
\frac{n}{m}\ge C_1\frac{n^{(1+\g)c_p}\sqrt{(1+\g)\log n}}{(j-1)^{\frac{1-p}{2}}},
\end{equation}
then
\[
\left\lbrace d_{n^{1+\gamma}}(i)\ge \frac{n}{m}\right\rbrace \subset \left\lbrace d_{n^{1+\gamma}}(i)\ge C_1\frac{n^{(1+\g)c_p}\sqrt{(1+\g)\log n}}{(j-1)^{\frac{1-p}{2}}}\right\rbrace.
\]
But \eqref{e:ifjnm} is always valid for large~$n$, since $(1+\gamma)c_p < 1$ and $j \ge m^{2/(1-p)}+1$. Since $i > (j-1)m$, Proposition~\ref{prop:upperboundeg} implies
\[
\mathbb{P}\left( T_{j,n}^{(m)} \le n^{1+\gamma}\right) \le mn^{-100(1+\g)}.
\]
Consequently
\begin{equation*}
  \sum_{n=1}^k err(n)=\!\!\!\!\!\!\!\!\sum_{n=\left[ (j-1)m+1\right]^{\frac{1}{1+\gamma}}}^k \!\!\!\!\!\!\!\!  err(n) \le \frac{m}{[(j-1)m]^{99}},
\end{equation*}
which concludes the proof.
\end{proof}

\begin{lema} \label{lemma:momentoR} For any vertex $j$ and all $m,R \in \mathbb{N}$, there exists a positive constant $c=c_{m,R,p}>0$ such that
\[ \Ed[T_{jm,1}^R] \le c j^R. \] 
\end{lema}
\begin{proof}
Note that we can write
\[
T_{jm,1}=1+\sum_{i=1}^{jm-1}(T_{i+1,1}-T_{i,1}) ,
\]
so that $T_{jm,1}$ is distributed as $1$ plus a sum of $jm-1$ independent geometric random variables of parameter $p$. Recall that a random variable which follows a negative binomial distribution of parameters~$jm-1$ and~$p$ has moment generating function
\[
\mathcal{G}(s) = \frac{(1-p)^{jm-1}}{(1-pe^s)^{jm-1}}.
\]
By taking the $R$-th derivative of $\mathcal{G}(s)$ and evaluating it at $0$, one can conclude the Lemma's statement.
\end{proof}
Now  we state and prove the main theorem of this section.
\begin{teorema}[Lower bound for the degree] \label{teo:cotainf} Fix $m$ sufficiently large, and let \[1 < R < mc_p(1-\d_m).\] Then there exists a positive constant $c = c(m, R, p)$ such that, for \[\beta\in(0,c_p(1-\d_m)) \textit{ and } j \geq m^{\frac{2}{1-p}} + 1\] we have:
\[ 
\P\left( d_{t,m}(j) < t^{\b} \right) \le c\frac{j^R}{t^{R-\b R/c_p(1-\d_m)}} +\frac{m}{[(j-1)m]^{99}} .
\]
\end{teorema}
%
\begin{proof} Fix $m$ sufficiently large. By Lemma ~\ref{lema:major}, 
\begin{equation}
\label{eq:startdeg}
\P(d_{t,m}(j) < t^{\b}) = \P(T_{j,t^{\b}}^{(m)} > t) \le \P\left(T_{j,m}^{(m)} \exp\left(\sum_{i=m}^{t^{\b}} \eta_i \right) > t \right) +  \frac{m}{[(j-1)m]^{99}}.
\end{equation}
We need to control the first term of the inequality's right hand side. We have that
\begin{equation} \label{eq:probexptilde}
\P\left(T_{j,m}^{(m)}\exp\left(\sum_{i=m}^{t^{\b}} \eta_i\right) > t\right) = \sum_{n = 1}^{\infty}\P\left(  \exp \left( \sum_{i = m}^{t^{\b}} \eta_i \right) > t/n \right)\P\left(T_{j,m}^{(m)} = n \right),
\end{equation}
because $T_{j,m}^{(m)}$ is independent of $\eta_i$ for all possible values of $i$. Since $\eta_i \eqd \Exp(ic_p(1-\d_i))$, we have that its moment generating function is given by
\[
\mathcal{G}_i(s) = \frac{1}{1-\frac{s}{ic_p(1-\d_i)}},
\]
for $s<ic_p(1-\d_i)$. Then, for $1 < R < mc_p(1-\d_m)$, Markov's Inequality implies
\begin{equation}
	\label{eq:expetabound}
	\begin{split}
		\P\left(\exp \left(\sum_{i=m}^{t^{\b}} \eta_i \right) > t/n \right) & 
		\le \frac{n^R}{t^R} \prod_{i=m}^{t^{\b}} \left( 1-\frac{R/[c_p(1-\d_i)]}{i}\right)^{-1} \\
		& \le \frac{n^R}{t^R} \prod_{i=m}^{t^{\b}} \left( 1-\frac{R/[c_p(1-\d_m)]}{i}\right)^{-1}.
	\end{split}
\end{equation}
The last product can be written in terms of the Gamma Function, using its multiplicative property. Note that 
\[
\prod_{i=m}^{t^{\b}}\left(1-\frac{R/c_p(1-\d_m)}{i} \right) = \frac{\G(m)\G\left(t^{\b}+1 - R/c_p(1-\d_m)\right)}{\G\left(m-R/c_p(1-\d_m)\right)\G(t^{\b}+1)}.
\]
This in turn implies the existence of a constant $b=b_{m,R,p}>0$ such that \[
\prod_{i=m}^{t^{\b}}\left(1-\frac{R/c_p (1-\d_m)}{i} \right) >  b t^{- \b R (c_p(1-\d_m))^{-1}}.
\]
Then, combining this bound with inequality \eqref{eq:expetabound}, we obtain 
\begin{equation}\label{eq:expR}
\P\left(\exp \left(\sum_{i=m}^{t^{\b}} \eta_i \right) > t/n \right) \le b \frac{n^Rt^{\b R(c_p(1-\d_m))^{-1}}}{t^R}.
\end{equation}
Notice that, by the time that the $jm$-th vertex enters the graph, the $j$-th block of $m$ vertices has total degree at least $m$, that is,  $T_{j,m}^{(m)} \le T_{jm,1}$. This fact, together with \eqref{eq:probexptilde}, Lemma~\ref{lemma:momentoR}, and the above inequality, implies
\begin{equation}
\label{eq:finaldeg}
\begin{split}
\P\left(T_{j,m}^{(m)}\exp\left(\sum_{i=m}^{t^{\b}} \eta_i\right) > t\right) 
&\le
\frac{b }{t^{R-\b R (c_p(1-\d_m))^{-1}}}\sum_{n=1}^{\infty}n^R\P\left(T_{j,m}^{(m)} = n \right) \\
&\le \frac{ b}{t^{R-\b R (c_p(1-\d_m))^{-1}}}\sum_{n=1}^{\infty}n^R\P\left(T_{jm,1} = n \right) \\
&=\frac{b }{t^{R-\b R (c_p(1-\d_m))^{-1}}}\Ed \left[T_{jm,1}^R \right]\\
&\!\!\!\!\!\!\!\stackrel{\tiny Lemma~\ref{lemma:momentoR}}{\leq} \frac{c j^R}{t^{R-\b R (c_p(1-\d_m))^{-1}}},
\end{split}
\end{equation}
for some constant $c=c_{m,p,R}>0$. Combining \eqref{eq:finaldeg} with \eqref{eq:startdeg} gives the desired result.
\end{proof}

\section{Communities in $G_t$} \label{sec:comu}
We are finally able to prove the main result of this paper: the existence of a community in~$G_t$ whose size grows to infinity polynomially in~$t$.
\begin{proof}[Proof of Theorem \ref{teo:community}] Fix some real number~$\e$ such that $0<\e<1$. Let $\a = (1-\e)\frac{(1-p)}{2-p}$, fix $\b = \frac{1+\e 2^{-1}(1-p)}{2}$ and choose $\e'>0$ such that $\e' < \alpha$. These choices of parameter imply that
\begin{equation}
\label{e:betabound}
\b:= \frac{1+\e 2^{-1}(1-p)}{2}<(1-\delta_m)c_p,
\end{equation}
and
\begin{equation}
\label{e:alfabound}
\frac{{\beta}}{(1-\delta_m)c_p}+\left( 1+\frac{1}{R} \right)\a=
(2-p)^{-1}\left(   1+\frac{\e(1-p)}{2(1-\delta_m)}+(1-p)(1-\e)+\frac{(1-p)}{R}(1-\e)   \right) <1,
\end{equation}
for sufficiently large~$m$ and~$R$. By Theorem~\ref{teo:cotainf} and the union bound we have that
\begin{equation}\label{eq:cota}
\P \left( \bigcup_{j=t^{\e'}}^{t^{\a}} \left\lbrace d_{t,m}(j) < t^{\b} \right\rbrace \right) \le \frac{C}{t^{R-\frac{\b R}{c_p(1-\d_m)} -\a(R+1)}}+\frac{m}{(mt)^{98\e'}}.
\end{equation}
Equations~\eqref{e:betabound} and~\eqref{e:alfabound} then imply that the right hand side of the above inequality goes to $0$ as $t$ goes to infinity. The immediate consequence of this fact is that inside each one of those $t^{\a} - t^{\e'}$ blocks of vertices of size $m$ there exists at least one vertex whose degree is larger than $t^{\b}/m$, with high probability. We denote by $L_j=L_{j,m,p}$ the vertex of largest degree among the vertices of the $j$-th block, and we call $L_j$ the \textit{leader} of its block. As we noted, $d_t(L_j) \ge t^{\b}/m$. 

We now prove that these vertices of large degree are connected with high probability. We write $i \nleftrightarrow j$ to denote the fact that there are no edges between the vertices $i$ and~$j$. We define the event
\[
B_t = \bigcup_{j=t^{\e'}}^{t^{\a}} \left\lbrace d_{t,m}(j) < t^{\b} \right\rbrace . 
\]
Let $g_s(i,j)$ be the indicator function of the event where we add an edge between $L_i$ and $L_j$ in an \textit{edge-step} at time $s$. We define the random variable
\[
Y_{2t}^{ij} = \prod_{s=t+1}^{2t} (1-g_s(i,j)).
\]
In other words, $Y_{2t}^{ij}$ is the indicator function of the event where we don't connect $L_i$ and $L_j$ in any of the \textit{edge-steps} between times~$t+1$ and~$2t$. We have that
\begin{equation}\label{eq:g}
\begin{split}
\Ed\left[\mathbb{1}_{B_t^c}(1-g_{2t}(i,j))  \middle| \GG_{2t-1} \right] & = \left(1 - 2(1-p)\frac{d_{2t-1}(L_i)d_{2t-1}(L_j)}{4(2t-1)^2} \right)\mathbb{1}_{B_t^c} \\
& \le \left(1 - (1-p)\frac{t^{2\b}}{8t^2 m^2} \right)\mathbb{1}_{B_t^c}.
\end{split}
\end{equation}

Now, observe that $Y_{2t}^{ij} = Y_{2t-1}^{ij}(1-g_{2t}(i,j))$. So, by using (\ref{eq:g}), we obtain
\begin{equation}
\begin{split}
\Ed\left[\mathbb{1}_{B_t^c} Y_{2t}^{ij} \middle| \GG_{2t-1} \right] & \le \left(1 - (1-p)\frac{t^{2\b}}{8t^2 m^2} \right)Y_{2t-1}^{ij}\mathbb{1}_{B_t^c}.
\end{split}
\end{equation}
Proceeding inductively, taking the conditional expectation with respect to $\GG_{s-1}$ at each step~$s$, we gain the following inequality
\begin{equation}\label{eq:expec}
\begin{split}
\Ed\left[\mathbb{1}_{B_t^c} Y_{2t}^{ij} \middle| \GG_{t} \right] & \le \left(1 - (1-p)\frac{t^{2\b}}{8t^2 m^2} \right)^t\mathbb{1}_{B_t^c} \le \exp \left( -c_1t^{2\b -1} \right)\mathbb{1}_{B_t^c},
\end{split}
\end{equation}
where $c_1$ is a positive constant depending on both $p$ and $m$. Since $\mathbb{1}_{\{L_i \nleftrightarrow L_j \textit{ in }G_{2t}\}} \le Y_{2t}^{ij}$, the above inequality implies
\[
\P \left(L_i \nleftrightarrow L_j \textit{ in }G_{2t}, B_t^c \right) \le \exp \left( -c_1t^{2\b -1} \right)\P\left(B_t^c\right).
\]
Now, by the union bound, we have
\[
\P \left( \bigcup_{t^{\e'}\le i,j\le t^{\a}} \left\lbrace L_i \nleftrightarrow L_j \textit{ in }G_{2t} \right\rbrace, B_t^c \right) \le t^{2\a}\exp \left( -c_1t^{2\b -1} \right)\P\left(B_t^c\right).
\]
And finally,
\begin{equation*}
\begin{split}
\P \left( \bigcup_{t^{\e'}\le i,j\le t^{\a}} \left\lbrace L_i \nleftrightarrow L_j \textit{ in }G_{2t} \right\rbrace \right) & = \P \left( \bigcup_{t^{\e'}\le i,j\le t^{\a}} \left\lbrace L_i \nleftrightarrow L_j \textit{ in }G_{2t} \right\rbrace, B_t^c \right) \\&\quad+ \P \left( \bigcup_{t^{\e'}\le i,j\le t^{\a}} \left\lbrace L_i \nleftrightarrow L_j \textit{ in }G_{2t} \right\rbrace, B_t \right) \\ 
& \le t^{2\a}\exp \left( -c_1t^{2\b -1} \right)\P\left(B_t^c\right) + \P\left(B_t\right).
\end{split}
\end{equation*}
The above inequality, together with our choice of $\a$, $\b$, and inequality~\eqref{eq:cota}, imply the existence of a subgraph of $G_t$ with order $t^\a(1-o(1))$ asymptotically almost surely.
\end{proof}

{\bf Acknowledgements } 
C.A.  was supported by Funda\c{c}\~{a}o de Amparo \`{a} Pesquisa do Estado de S\~{a}o Paulo (FAPESP), grant 2013/24928-2. R.S. has been partially supported by Conselho Nacional de Desenvolvimento Cient\'\i fico e Tecnol\'{o}gico (CNPq)  and by FAPEMIG (Programa Pesquisador Mineiro), grant PPM 00600/16. R.R. has been partially supported by Coordena\c{c}\~{a}o de Aperfei\c{c}oamento de Pessoal de N\'{i}vel Superior (CAPES).

\bibliographystyle{plain}
\bibliography{ref}

\end{document}